\documentclass[12pt]{article}
\usepackage{amsmath,amssymb,amsfonts}
\usepackage{graphicx}
\usepackage{amsthm}
\usepackage{xcolor}

\newtheorem{theorem}{Theorem}

\newtheorem{remark}{Remark}

\newtheorem{obs}{Observation}
\newtheorem{defin}{Definition}

\title{Adjacency posets of outerplanar graphs}

\author{Marcin Witkowski\footnote{Faculty of Mathematics and Computer Science, Adam Mickiewicz University, Poznan, Poland, mw@amu.edu.pl}}

\begin{document}

\maketitle

\begin{abstract}
Felsner, Li and Trotter showed that the dimension of the adjacency poset of an outerplanar graph is at most 5, and gave an example of an outerplanar graph whose adjacency poset has dimension 4. We improve their upper bound to 4, which is then best possible.
\end{abstract}

\section{Introduction}

With a graph $G$ we can associate the poset $A_G$ in the following way:

\begin{defin}[Adjacency poset \cite{def}]
For a graph $G(V,E)$, the \textit{adjacency poset} $A_G$ is a poset of height $2$ with minimal elements $x$ and maximal elements $\overline{x}$ corresponding to each vertex $x \in V$, and comparabilities $x < \overline{y},\ y < \overline{x}$ for each $xy \in E$.
\end{defin}

There is also another, more common \cite{adj,sch}, way to define a poset from a graph $G(V,E)$. The \textit{incidence poset} $P_G$ is a poset of height $2$ with minimal elements as vertices and maximal elements as edges with $x<e$ if $x$ is incidence to $e$. Further generalizations of this concept, to posets of vertices, edges and faces ordered by inclusion were also investigated (see \cite{bb1,bb2} for reference). 

Let $P=(X,\le_P)$ be a partial order, a total order $L=(X,\le_L)$ is a \emph{linear extension} of $P$ if
$x \le_P y$ implies $x \le_L y$. If $x$ and $y$ are incomparable in $P$, i.e., neither $x \le_P y$ nor $y \le_P x$, and $y \le_L x$, then we
say that ordered pair $(x,y)$ is \emph{reversed} in $L$.

\begin{defin}[Poset dimension]
The \textit{dimension} of a partially ordered set $P$, denoted by $dim(P)$, is the smallest $n$ such that $P = \bigcap^n_{i=1} \mathcal{L}_i $, where $\mathcal{L}_i$ are linear extensions of $P$. The set $\{\mathcal{L}_i\}$ is called a realizer of $P$.
\end{defin}

In other words, the dimension of $P$ is the smallest size of a collection
of linear extensions of $P$ such that every (ordered) incomparable pair $(x,y)$
is reversed in one of the linear extensions. 

Given $P=(X,\le_P)$, an incomparable pair $(x,y)$ is called a \emph{critical pair} when $z<x$ in $P$ implies $z<y$ in $P$, for all $z \in X$, and $w > y$ in $P$ implies $w > x$ in $P$, for all $w \in X$. Adding a private neighbour to each vertex of
the outerplanar graph $G$ ensures that all relevant critical
pairs of the adjacency poset $A_G$ are pairs of minimal-maximal elements. In \cite{adj} it is shown that a realizer that reverts all critical pairs of $A_G$ realize the poset $A_G$. 

The \textit{standard example} $S_n$ of a poset of dimension $n > 2$ is a height $2$ poset consisting of $n$ minimal elements $\{x_1,...,x_n\}$ and $n$ maximal elements $\{y_1,...,y_n\}$, with $x_i \leq y_j$ if and only if $i \neq j$. The dimension of $S_n$ is equal to $n$.

Schnyder proved in \cite{sch} that graph is planar if and only if its incidence poset has dimension at most $3$ showing a non-trivial connection between incidence posets dimension and graph properties. Felsner, Li and Trotter \cite{adj} proved the following results for the case of adjacency posets of planar graphs.

\begin{theorem}
If $G$ is a planar graph $G$, then $dim(A_G)\leq~8$. Furthermore, there exists a planar graph whose adjacency poset has dimension at least $5$.
\end{theorem}  

When it comes to outerplanar graphs, we know that dimension of incidence poset of the outerplanar graph with no vertices of degree $1$ is at most $\left[ 2\updownarrow 3\right]$. Which means that we can find three linear orders, in which two are reverse to each other, so that they realize this poset (see \cite{outer} for more details). 

The by now best bound for the dimension of adjacency posets of outerplanar graphs was given by Felsner, Li and Trotter \cite{adj}.

\begin{theorem}
If $G$ is an outerplanar graph $G$, then $dim(A_G) \leq 5$. Furthermore, there exists an outerplanar graph whose adjacency poset has dimension at least $4$.
\end{theorem}  

In the general case, we know that for every non-negative integer $g$, there exist an integer $d_g$ such that the dimension of the adjacency poset of a graph of genus $g$ is at most $d_g$. Felsner, Li and Trotter \cite{adj} used acyclic colourings to show this fact obtaining a bound of about $O(g^{8/7})$, and conjecture that dimension should be bounded by $O(g)$. This conjecture was later proven to be true, and then improved by several authors \cite{aaa,bbb, scott}. 

\begin{theorem}
If $G$ is a graph of genus $g$ and $A_G$ is its adjacency poset, then $dim(A_G) \leq 26 \sqrt{g \log g}  + \chi(G) + 7$.
\end{theorem}

More results and tools on poset dimension can be found in Trotter's monograph  \cite{book} and survey \cite{surv}.

The proof of the bound for planar graphs and outerplanar graphs from \cite{adj} uses Schnyder woods and construction of a particular vertex and edge colouring in planar triangulations. In this paper, we show that $4$ linear orders are enough to realize the adjacency poset of an outerplanar graph. Our proof is more straightforward and uses only the elementary properties of embedding of an outerplanar graph in a plane. Moreover, for completeness of the argument, we provide a lower bound example stated in \cite{adj} (the proof of the bound is also slightly different than the one given by Felsner, Li and Trotter).

\section{Basic properties}

Bipartite planar graphs are the only graphs for which we know strict bounds for the dimension of their adjacency posets \cite{adj}. In general, we do know that this parameter can be bounded from below by the chromatic number of a graph.

\begin{obs}Dimension of the adjacency poset of a graph $G$ is at least the chromatic number of a graph.
\end{obs}

\begin{proof}
The vertex set of pairs of minimal and maximal elements ($x$,$\overline{x}$) that can be reversed in a single linear order forms an independent set in a graph $G$. Therefore the dimension is bounded from below by the chromatic number.
\end{proof}

On the other hand, there is no function of the chromatic number of a graph that can be taken as an upper bound for the dimension of adjacency poset of a graph.

\begin{obs}Dimension of the adjacency poset of a graph $G$ can be arbitrarily larger than the chromatic number of a graph.
\end{obs}

\begin{proof}
If $G$ is a bipartite graph, then $A_G$ can be seen as two disjoint (oriented) copies of $G$.
Hence, if $G$ is $K_{n,n}$ minus a perfect matching, then $A_G$ consists of
two copies of the standard example $S_n$. In this case $\chi(G) = 2$ and $dim(A_G) =~n$.
\end{proof}

\begin{theorem}
There exists an outerplanar graph $G$ which adjacency poset has dimension at least $4$. 
\end{theorem}

\begin{center}
\begin{figure}[htb]
\centering
\includegraphics[width = 280pt]{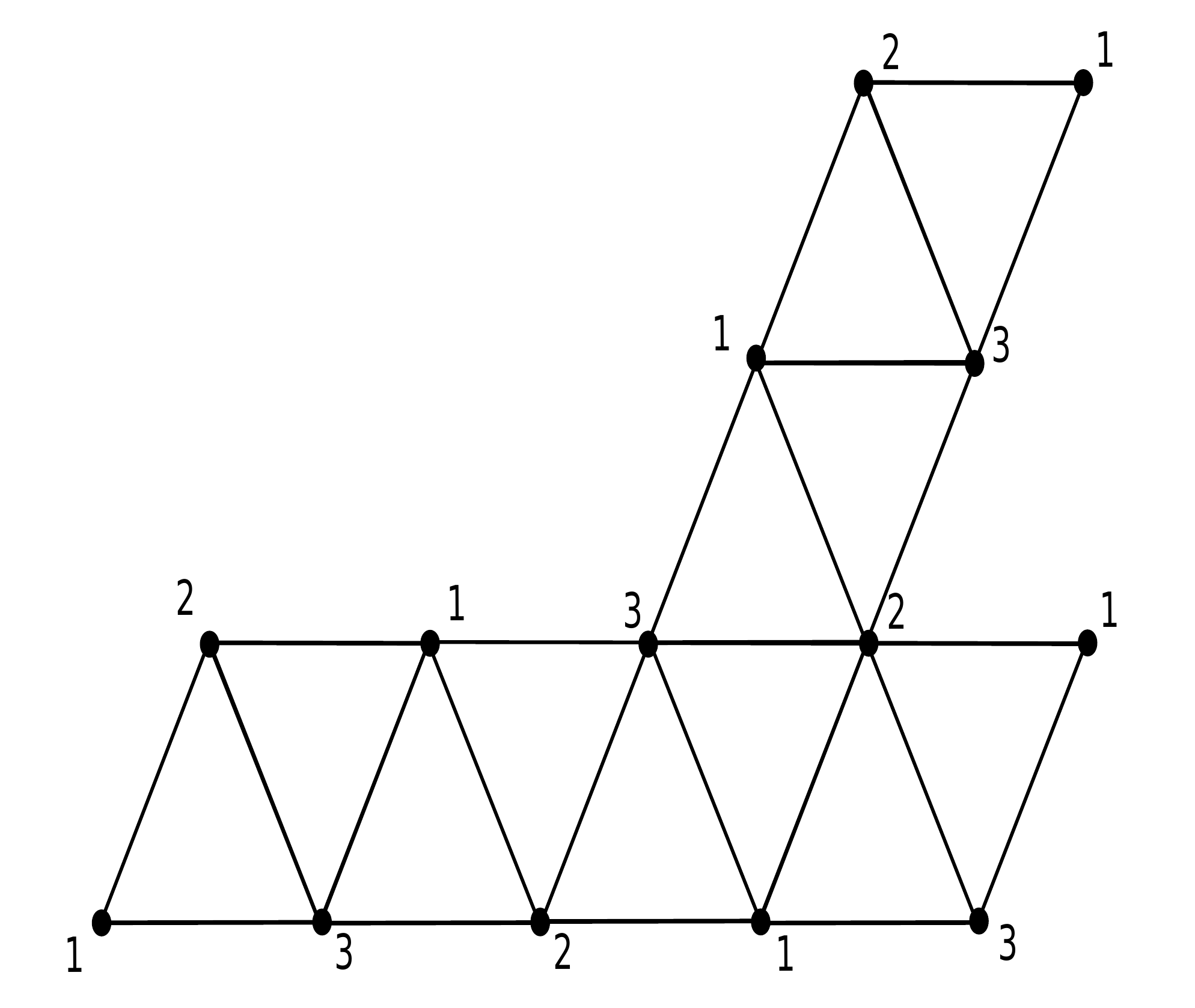}
\caption{$H$ is outerplanar and its adjacency
poset has dimension 4.}
\label{figure1}
\end{figure}
\end{center}

\begin{center}
\begin{figure}[ht]
\centering
\includegraphics[width = 200pt]{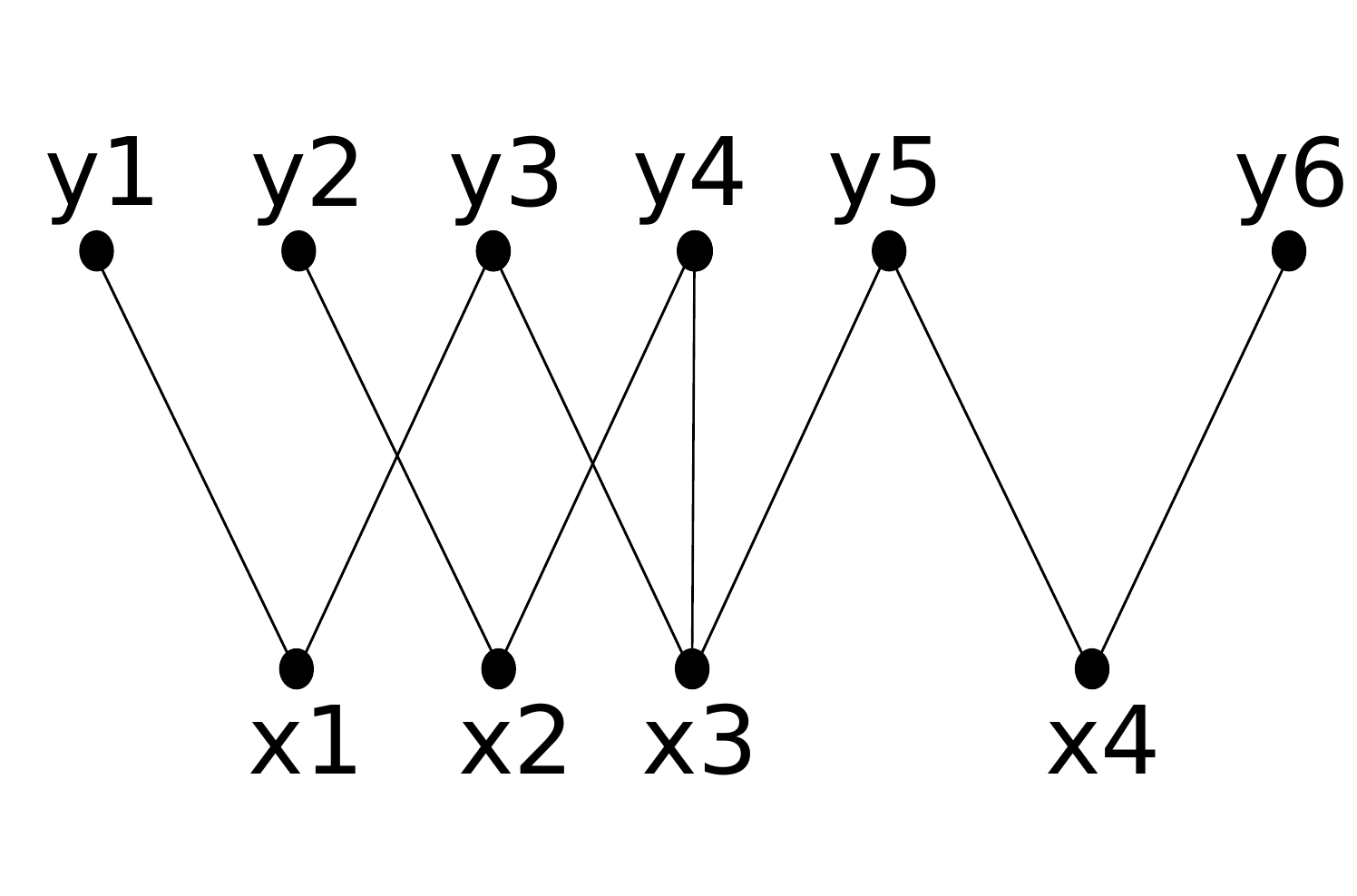}
\caption{Subposet spanned by vertices of colour $1$ and $3$. Bottom part are the vertices of colour $1$ and top are vertices of colour $3$.}
\label{figure2}
\end{figure}
\end{center}

\begin{proof}
Consider the graph $H$ from Figure \ref{figure1}, with given (unique up to permutation of colors) $3$-colouring (with colour sets $X,Z,Y$ corresponding to colors $1,2,3$) and its adjacency poset $A_H$. Notice that every triangle in $H$ defines a standard example $S_3$ as a subposet in $A_H$. Thus, we can reverse only one pair of vertices $(v,\overline{v})$ in a single linear extension for every triangle. Therefore if the dimension of $P$ would be equal $3$, we have to reverse exactly one pair $(v,\overline{v})$ from every triangle in a single total order. The only way to do it is to reverse all pairs $(v,\overline{v})$ that get the same colour in $H$ in a single linear extension. Moreover every pair $(v,\overline{u})$ where $v,u \in Z$, have to be reversed in a single linear extension as every pair of vertices with colour $2$ share a common neighbour with colour $1$ and $3$ and therefore can not be reversed in the linear order in which those vertices are reversed. 
The previous statement implies that a subposet defined on pairs of vertices from $X$ as minimal elements and $Y$ as maximal elements needs to have dimension at most $2$. 

Poset from Figure \ref{figure2} is one of the examples identified by Felsner \cite{3fe} as being 3-interval irreducible and therefore having a dimension greater than $2$. For the completeness of the argument here we include a short proof that given poset can not have dimension $2$.

Notice that $(x1,x4,y3,y5)$ creates $S_2$, so can not be reversed in a single linear order. If we reverse them in two separate orders then wlog they will look as follows (by $|x3,x4|$ we mean that the order of those two elements is not defined):

\begin{itemize}
\item[$\mathcal{L}_1 :$]{ $|x3,x4| < y5 < x1 < |y1,y3|$  }
\item[$\mathcal{L}_2 :$]{ $|x1,x3| < y3 < x4 < |y5,y6|$  }
\end{itemize}

\noindent We also need to reverse pairs ($x3,y1$), ($x3,y6$), hence, as we know that $x3 < y1$ in $\mathcal{L}_1$ and $x3 < y6$ in $\mathcal{L}_2$ we need to reverse them in the other order. Thus

\begin{itemize}
\item[$\mathcal{L}_1 :$]{ $ x4 < y6 < x3 < y5 < x1 < |y1,y3|$  }
\item[$\mathcal{L}_2 :$]{ $ x1 < y1 < x3 < y3 < x4 < |y5,y6|$  }
\end{itemize}

\noindent Next, we need to put $y4$ above $x3$ but reverse pairs $(x1,y4)$, $(x4,y4)$, which gives us

\begin{itemize}
\item[$\mathcal{L}_1 :$]{ $ x4 < y6 < x3 < |y5,y4| < x1 < |y1,y3|$  }
\item[$\mathcal{L}_2 :$]{ $ x1 < y1 < x3 < |y3,y4| < x4 < |y5,y6|$  }
\end{itemize}

\noindent Now $x2$ has to be below $y4$, but we still need to reverse pairs $(x2,y1)$, $(x2,y3)$, $(x2,y5)$, $(x2,y6)$. Thus

\begin{itemize}
\item[$\mathcal{L}_1 :$]{ $ x4 < y6 < x3 < y5 < x2 < y4 < x1 < |y1,y3|$  }
\item[$\mathcal{L}_2 :$]{ $ x1 < y1 < x3 < y3 < x2 < y4 < x4 < |y5,y6|$  }
\end{itemize}

We get that in both linear orders $x3 < x2$, so we can not reverse pair $(x3,y2)$, contradiction.

Finally, that means that the adjacency poset of graph $H$ has dimension at least $4$.

\end{proof}

\begin{remark}
There exists a planar graph which adjacency poset has dimension at least $5$.
\end{remark}

\begin{proof}
Add an apex $v$ to an outerplanar graph $H$ with
$dim(A_H) = 4$. The dimension of the subposet defined on $H$ is $4$. Moreover when we reverse pair $(r,\overline{r})$, then we can not reverse any other pair of vertices from different levels (as $r$ was connected to every vertex in $H$). Therefore the dimension of $A_H$ is at least $5$.
\end{proof}

\section{Main theorem}

In this section, we present the proof of the main theorem.

\begin{theorem}
The dimension of the adjacency poset of an outerplanar graph is at most 4
\end{theorem}

\begin{proof}
Let $G$ be an outerplanar graph. We fix a 3-colouring of the graph $G$ and embedding of the graph $G$ into plane such that all vertices of $G$ lie on a single line and no edges are crossing each other. Notice that embedding creates a natural order of vertices on the line from the left end to the right end. Denote the colour sets by $A, B$ and $C$ and adjacency poset of $G$ as $A_G$. 

 An adjacency poset is a poset of height two. According to colour classes we denote by $A, B, C$ minimal elements of the poset $A_G$ and by $\overline{A}, \overline{B}, \overline{C}$ maximal elements of the poset $A_G$.

We construct a family $\mathcal{R} = \{\mathcal{L}_1; \mathcal{L}_2; \mathcal{L}_3; \mathcal{L}_4\}$ of four  linear extensions as follows (see Figure \ref{figure3} for example): Let X be a subset of minimal elements and Y be a subset of maximal elements. We define linear extensions of $A_G$ on those sets in the following way:
\begin{itemize}
\item{$(X,Y)^\rightarrow$ to be a linear order on the set $X \cup Y$ created as follows: $X$ preserves the order from the embedding of $G$ and elements from $Y$ are assigned to the first place from the beginning of the order that satisfies all comparabilities. If more than one element of $Y$ can be ordered in the same position assign all of them in an order reverse to order of their corresponding copies in embedding.}

\item{$(X,Y)^\leftarrow$ to be a linear order on the set $X\cup Y$ created as follows: $X$ are assigned in an order reverse to the one from the embedding of $G$ and elements from $Y$ are assigned to the first place from the beginning of the order that satisfies all comparabilities. If more than one element of $Y$ can be ordered in the same position assign all of them in order of their corresponding copies in embedding.}

\item{$(X,Y)^\triangle$ (where $Y$ is a monochromatic set in $G$) to be a linear order on the set $X\cup Y$ created as follows: Let $I_{v}$ be the open interval of points in the embedding belonging to the region bounded by the edges incidence to $v \in V(G)$. Notice that by planarity of embedding if $v \in V(G)$ is not connected to $u \in V(G)$ then $I_{u}$ and $I_{v}$ are either disjoint or include in one another. For minimal elements $v \in X$ let $J_v = I_v \cup \bigcup_{u\in Y\cap N(v)} I_u$ (to be the sum of interval $I_v$ and intervals of all its maximal neighbours). Let $P(X,Y)$ be the reverse inclusion order on $ \{J_v : v\in X\} \cup \{I_u : u\in Y\}$ where in case of a tie $J_v = I_u$ we set $J_v < I_u$ and ties between $I_u,I_v$ and $J_u,J_v$ are resolved by the order of embedding. Now let $L$ be the ordering on $X\cup Y$ obtained from a linear extension of $P(X,Y)$ by replacing the sets by their defining vertices.
Notice that the order preserve comparabilities between $X$ and $Y$ as for an edge $\{xy\in E(G) : x \in X, y \in Y\}$, we have $I_y \subseteq J_x$, hence $x<y$.}
\end{itemize}

\begin{center}
\begin{figure}[ht]
\centering
\includegraphics[width = \textwidth]{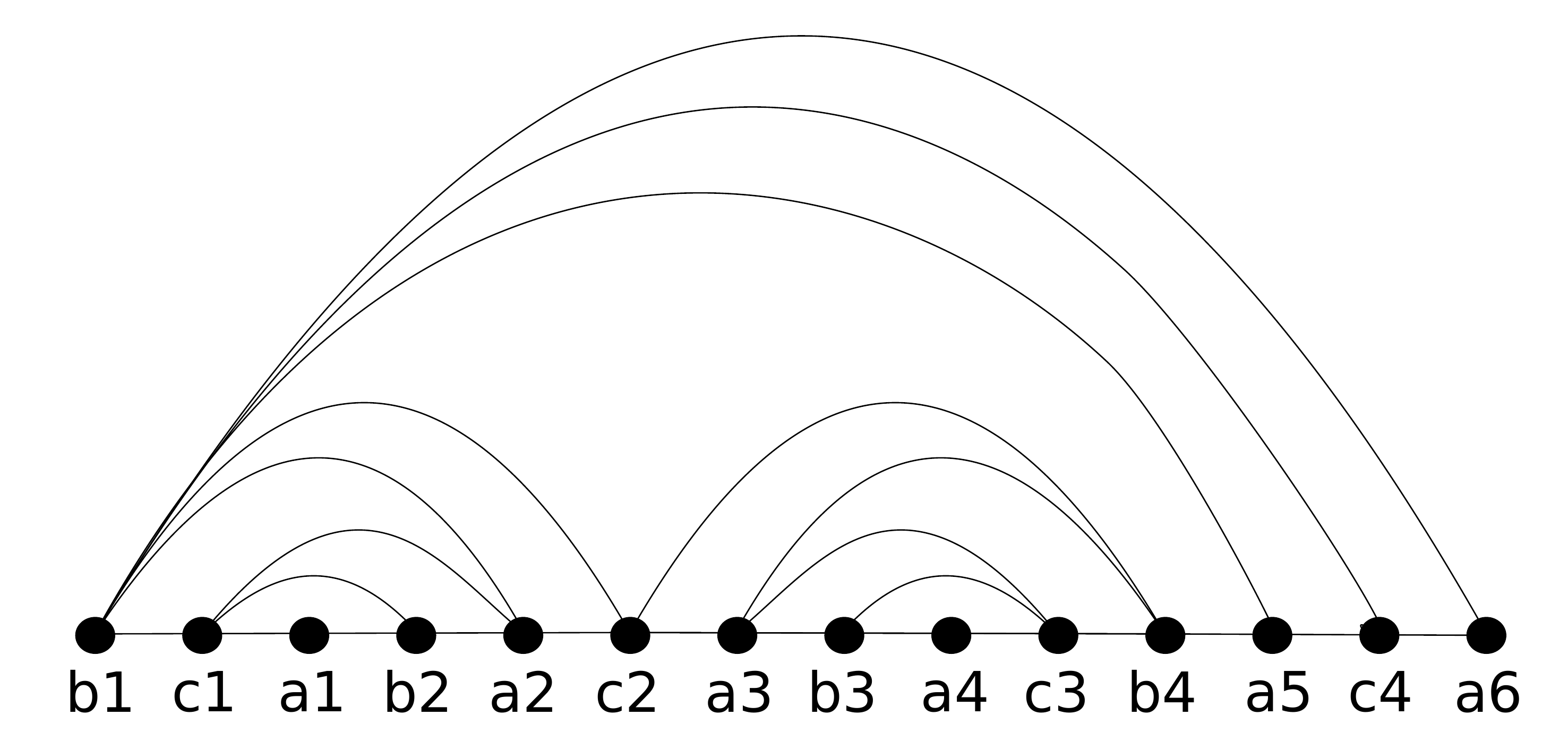}
\caption{Embedding of the graph from Figure 1, example of orders: 
\newline $ {(A,{\overline{B}\cup \overline{C}})}^\rightarrow : a1\ a2\ \overline{b2}\ \overline{c1}\ a3\ \overline{c2}\ a4\ \overline{c3}\ \overline{b3}\ a5\ \overline{b4}\ a6\ \overline{c4}\ \overline{b1}$
\newline $ {(A,{\overline{B}\cup \overline{C}})}^\leftarrow : a6\ a5\ \overline{c4}\ a4\ a3\ \overline{b4}\ \overline{c3}\ \overline{b3}\ a2\ \overline{c2}\ \overline{b1}\ a1\ \overline{b2}\ \overline{c1}$
\newline
$ {({B \cup C},{\overline{A}})}^\triangle :\ b1\ c4\ \overline{a6}\ b4\ \overline{a5}\ c2\ c1\ b2\ \overline{a2}\ \overline{a1}\ b3\ c3\ \overline{a3}\  \overline{a4}  $
}
\label{figure3}
\end{figure}
\end{center}

We claim that following realizer $\mathcal{R}$ with $4$ linear orders realizes the adjacency poset of every outerplanar graph $G$.

\begin{center}
\begin{tabular}{c l}
 $\mathcal{L}_1:$ & $ {({B \cup C},{\overline{A}})}^\triangle {(A,{\overline{B}\cup \overline{C}})}^\rightarrow$ \\
 $\mathcal{L}_2:$ & $ ({A \cup C},{\overline{B}})^\triangle (B,{\overline{A}\cup \overline{C}})^\rightarrow$ \\
 $\mathcal{L}_3:$ & $ ({A \cup B},{\overline{C}})^\triangle (C,{\overline{A}\cup  \overline{B}})^\rightarrow$ \\
 $\mathcal{L}_4:$ & $ (A\cup B\cup C,{\overline{A}\cup \overline{B}\cup \overline{C}})^\leftarrow$ \\
 \\
\end{tabular}
\end{center}

Notice that: 

a) Vertices within sets $A,\overline{A},B,\overline{B},C,\overline{C}$ as well as pairs of vertices from sets $(A,B)$, $(A,C)$,$(\overline{A},\overline{B})$, $(\overline{A},\overline{C})$, $(B,C)$, $(\overline{B},\overline{C})$ are not a critical pairs in $A_G$.

\begin{proof}
If this is not true for $G$, we may add vertices with connecting edges to form a graph $H$ satisfying the assumption that $G$ is an induced subgraph of $H$. Consequently, the adjacency poset of $G$ will be an induced subposet of the adjacency poset of $H$. We know from \cite{adj} that reversing critical pairs is enough to obtain a realizer for $A_H$.
\end{proof}

b) Pairs of vertices from sets $(A,\overline{A})$ and  $(B,\overline{B})$ and $(C,\overline{C})$ are incomparable in $\mathcal{R}$.

\begin{proof}
First, two orders reverse $A$ with $\overline{A}$, second and third $B$ with $\overline{B}$, while first and third $C$ with $\overline{C}$.
\end{proof}

c) In realizer $\mathcal{R}$ the minimal elements of the poset $A_G$ are comparable with the maximal elements of $A_G$ from different colour classes iff there is an edge between them in $G$.

\begin{proof}
We prove that the statement is true for pairs of vertices $\overline{a} \in \overline{A}$ and $b \in B$. Assume that $a\in A$ is not connected to $b\in B$ in $G$. Notice that we have $\overline{a}>b$ in $\mathcal{L}_3$ and those two vertices lies in  ${({B \cup C},{\overline{A}})}^\triangle$ in $\mathcal{L}_1$ and in $(B,{\overline{A}\cup \overline{C}})^\rightarrow$ in $\mathcal{L}_2$. There are two cases: 

\begin{itemize}
\item[i)] {In given embedding of vertices of $G$ into a single line $b$ lies outside of $I_a$ (i.e. to the left of the leftmost or to the right of the rightmost neighbours of $a$).} \item[ii)] {In given embedding of vertices of $G$ into a single line $b$ lies inside $I_a$ (i.e. between leftmost and rightmost neighbours of $a$). This also implies that $I_b \subset I_a$.}
\end{itemize}

ad. i) In the first case either $a$ is not connected to any vertex from $b \in B$, and we have $\overline{a}<b$ in $\mathcal{L}_2$, if there is some $b'$ such that $ab' \in E(G)$, then if $b$ lies to the left from neighbourhood of $a$, then $\overline{a} < b$ in $\mathcal{L}_4$. If $b$ lies to the right from neighbourhood of $a$ then $\overline{a}<b$ in $\mathcal{L}_2$.

ad ii) In the second case if $b$ is not connected to any $a \in A$ then $\overline{a}<b$ in $\mathcal{L}_1$ (as $J_b = I_b \subset I_a$ ). 

In the case that there is an $a' \in A$ such that $a'b \in E(G)$ then $a'$ has to lie between leftmost and rightmost neighbour of $a$ as otherwise edge $a'b$ would cross with an edge incident to $a$ which is a contradiction with embedding of $G$ being planar. Moreover due to the same reason any neighbour of $a'$ have to lie between
leftmost and rightmost neighbour of $a$. Thus in ${({B \cup C},{\overline{A}})}^\triangle$ we will have $I_{a'} \subseteq J_b$, for all $a' \in N(b)$ but as $I_{b} \subset I_a$ and $I_{a'} \subset I_a$ and $a \notin J_b$ (because the set A is monochromatic) we have that $J_b \subset I_{a}$, which implies that $\overline{a}<b$. Thus pair ($\overline{a},b$) is incomparable in $\mathcal{L}_1$. That concludes the proof in the case $(a,b) \notin E(G)$.

On the other hand, if there is an edge between $a$ and $b$, notice that in our definition of linear orders we always put either $\overline{a}$ or $b$ in a way that preserves the comparabilities with the other set. Thus realizer will preserve the order between them.

Finally noticed that our linear orders are symmetrically defined so we can repeat the same argument for pairs of vertices from any two colour classes.
\end{proof}

We have shown $4$ linear orders that realize every outerplanar graph's adjacency poset, so this concludes the proof.

\end{proof}

\section{Final remarks}

In the case of outerplanar graphs, we can notice connections between adjacency and incidence posets. The proof method used to show that the dimension of outerplanar graphs is  $\left[ 2\updownarrow 3\right]$ use similar, in flavour, inclusion order corresponding to the graph's embedding into the plane \cite{outer}. 

We strongly believe that using two reverse orders in the realizer allows us to use only four orders in the adjacency poset's realizer. That is an indicator that for the planar graphs the lower bound of $5$ might not be correct, and it is more likely that $6$ is an actual value (as for incidence posets the realizer of a planar graph can not use two orders that are reverse to each other). 

\bibliographystyle{plain}


\begin{thebibliography}{9}

\bibitem{aaa} A. Adiga, D. Bhowmick, L.S. Chandran, "Boxicity and Poset Dimension", SIAM J. Discrete Math., 25(4) (2011), pp 1687–1698

\bibitem{bb1} G. R. Brightwell, W. T. Trotter, "The order dimension of convex polytopes", SIAM J.Discrete Math. 6, (1993), pp 230–245.

\bibitem{bb2} G. R. Brightwell, W. T. Trotter, "The order dimension of planar maps", SIAM J.Discrete Math. 10, (1997), pp 515–528.

\bibitem{bbb} L. Esperet, G. Joret, "Boxicity of graphs on surfaces", Graphs Combinatorics, 29(3), (2013), pp 417–427.

\bibitem{3fe} S. Felsner, "3-Interval irreducible partially ordered sets", Order, 11 (1994), pp 97-125

\bibitem{adj} S. Felsner, Ch. M. Li, W. T. Trotter, "Adjacency Posets of Planar Graphs", \textit{Discrete Mathematics} 310, (2010), pp 1097-1104.

\bibitem{def} S. Felsner, W. T. Trotter, "Dimension, graph and hypergraph coloring", \textit{Order} 17, (2000), pp 167–177.

\bibitem{outer} S. Felsner, W. T. Trotter, "Posets and planar graphs". J. Graph
Theory 49, (2005), pp 273–284. ISSN 0364-9024.

\bibitem{scott} A. D. Scott, D. Wood, "Better bounds for poset dimension and boxicity". Transactions of the American Mathematical Society. (Accepted/ In press) https://doi.org/10.1090/tran/7962

\bibitem{sch} W. Schnyder, "Planar graphs and poset dimension", \textit{Order} 5, (1989) 323–343.

\bibitem{book} W. T. Trotter, "Combinatorics and Partially Ordered Sets: Dimension Theory", The Johns Hopkins University Press, Baltimore, 1992.

\bibitem{surv} W. T. Trotter, "Partially ordered sets", in R.L.Graham, M.Grötschel, L.Lovász (Eds.), Handbook of Combinatorics, Elsevier, Amsterdam, 1995, pp. 433–480.


\end{thebibliography}

\end{document}